\documentclass[11pt]{amsart}
\usepackage[T1]{fontenc}
\usepackage{amsfonts}
\usepackage{verbatim}
\usepackage{amsmath, amsthm, amssymb}
\usepackage{natbib}

\numberwithin{equation}{section}
\allowdisplaybreaks

\newtheorem{lemma}{Lemma}[section]
\newtheorem{Lemma}{Lemma}[section]

\newtheorem{theorem}[Lemma]{Theorem}

\newtheorem{corollary}[Lemma]{Corollary}

\newcommand{\reals}{{\mathbb R}}
\newcommand{\bbr}{\reals}
\newcommand{\bbn}{{\mathbb N}}
\newcommand{\vep}{\varepsilon}

\newcommand{\bbc}{C(\reals)}

\newcommand{\BX}{{\bf X}}
\newcommand{\BY}{{\bf Y}}
\newcommand{\BZ}{{\bf Z}}

\newcommand{\bx}{{\bf x}}
\newcommand{\by}{{\bf y}}

\newcommand{\one}{{\bf 1}}

\def\ProbSpace{\bigl( \Omega, {\mathcal F}, P\bigr)}
\def\cadlag{c\`adl\`ag} 

\begin{document}

\title[Location of the supremum]{Distribution of the supremum location of stationary processes}

\author[G. Samorodnitsky]{Gennady Samorodnitsky}
\address{School of Operations Research and Information Engineering\\
and Department of Statistical Science \\
Cornell University \\
Ithaca, NY 14853}
\email{gs18@cornell.edu}
\author[Y. Shen]{Yi Shen}
\address{School of Operations Research and Information Engineering\\
Cornell University\\
Ithaca, NY 14853}
\email{ys437@cornell.edu}

\thanks{This research was partially supported by the ARO
grant W911NF-07-1-0078, NSF grant  DMS-1005903 and NSA grant
H98230-11-1-0154  at Cornell University.}

\keywords{stationary process, global supremum location, bounded variation, strong mixing
\vspace{.5ex}}

\setcounter{lemma}{0}


\begin{abstract}
The location of the unique supremum of
a stationary process on an interval does not need to be uniformly
distributed over that interval. We describe all possible distributions
of the supremum 
location for a broad class of such stationary processes. We show that,
in the strongly mixing case, this distribution does tend to the
uniform in a certain sense as the length of the interval increases to
infinity. 
\end{abstract}

\maketitle
\section{Introduction}\label{sec: Intro}

\baselineskip=18pt 

 Let $\BX = (X(t),\, t\in\bbr)$ be a sample continuous stationary
 process. Even if, on an event of probability 1, the supremum of the
 process over a compact interval $[0,T]$ is
 attained at a unique point, this point does not have to be uniformly
 distributed over that interval, as is known since
 \cite{leadbetter:lindgren:rootzen:1983}. However, its distribution
 still has to be 
 absolutely continuous in the interior of the interval, and the
 density has to satisfy very specific general constraints, as was
 shown in a recent paper \cite{samorodnitsky:yi:2011a}. 

 In this paper we give a complete description of the family of possible
 densities of the supremum location for a large class of sample
 continuous  stationary processes. The necessary conditions on these
 densities follow by combining certain general results cited above,
and for every function satisfying these necessary conditions we
construct a stationary process of the required type for which this 
function is the density of the supremum location. This is done in
Section \ref{sec:special}, which is preceded by Section
\ref{sec:assumptions} in which we describe the class of stationary
processes we are considering and quote the results from
\cite{samorodnitsky:yi:2011a} we need in the present paper. 
Next, we show that for a large class of stationary processes,  under a
certain strong mixing assumption, the distribution of
the supremum location does converge to the uniformity for very long
intervals, and it does it in a strong sense. This is shown in Section
\ref{sec:long.intervals}.

\section{Preliminaries} \label{sec:assumptions} 

For most of this paper $\BX = (X(t),\, t\in\bbr)$ is a
stationary process with continuous sample paths, defined on
a probability space $\ProbSpace$, but in Section 
\ref{sec:long.intervals} we will allow upper semi-continuous sample
paths. In most of the paper (but not in Section 
\ref{sec:long.intervals})  we will also impose two assumptions on the 
process, which we now state. 

For $T>0$ we
denote  by $X_*(T) = \sup_{0\leq t\leq T}X(t)$, the largest value of
the process in the interval $[0,T]$. 

\medskip\noindent
{\bf Assumption} U$_T$: 
$$
P\Bigl( X(t_i)=X_*(T), \, i=1,2, \ 
\text{for two different $t_1,t_2\in [0,T]$}\Bigr)=0.
$$

It is easy to check that the probability in Assumption U$_T$ is
well defined. Under the assumption, the supremum over
interval $[0,T]$ is uniquely achieved. 

The second assumption on a stationary process deals with the
fluctuations of its sample paths. 

\medskip\noindent
{\bf Assumption} L: 
$$
K:=\lim_{\vep\downarrow 0}\frac{P\bigl( \text{$\BX$ has a local
    maximum in $(0,\vep)$}\bigr)}{\vep}<\infty\,,
$$
with the limit easily shown to exist. 
Under Assumption L the process $\BX$ has sample paths
of locally bounded variation; see Lemma 2.2 in
\cite{samorodnitsky:yi:2011a}.

For a compact interval $[a,b]$, we will denote by 
$$
\tau_{\BX,[a,b]}= \min\bigl\{ t\in [a,b]:\ X(t) =\sup_{a\leq s\leq
  b}X(s)\bigr\} 
$$
the leftmost location of the supremum in the interval; it is a well
defined random variable.  If the supremum
is unique, the adjective ``leftmost'' is, clearly, redundant. For
$a=0$, we will abbreviate $\tau_{\BX,[0,b]}$ to $\tau_{\BX,b}$, and use
the same abbreviation in similar situations in the sequel. 

We denote by $F_{\BX,[a,b]}$ the law of $\tau_{\BX,[a,b]}$; it is a
probability measure on the interval $[a,b]$. It was proved in
\cite{samorodnitsky:yi:2011a} that  for any $T>0$ the probability
measure $F_{\BX,T}$ is absolutely continuous in the interior of the
interval $[0,T]$, and density can be chosen to be right continuous and
have left limits; we call this version of the density
$f_{\BX,[a,b]}$. This version of the density satisfies a universal
upper bound 
\begin{equation} \label{e:upper.bound}
f_{\BX,T}(t) \leq \max\left(\frac1t,\frac{1}{T-t}\right), \ 0<t<T\,.
\end{equation}

We will also use the following result from the above
reference. 
\begin{lemma} \label{l:key}
Let $0\leq \Delta<T$. Then for every $0\leq \delta\leq \Delta$, 
$f_{\BX,T-\Delta}(t)\geq f_{\BX,T}(t+\delta)$ almost everywhere in
  $(0,T-\Delta)$. Furthermore, for every such $\delta$ and every
  $\vep_1,\vep_2\geq 0$, such that $\vep_1+\vep_2<T-\Delta$, 
\begin{equation} \label{e:density.var.int}
\int_{\vep_1}^{T-\Delta-\vep_2}\bigl(
f_{\BX,T-\Delta}(t)-f_{\BX,T}(t+\delta)\bigr)\, dt 
\end{equation} 
$$
\leq
\int_{\vep_1}^{\vep_1+\delta} f_{\BX,T}(t)\, dt
+ \int_{T-\Delta-\vep_2+\delta}^{T-\vep_2} f_{\BX,T}(t)\, dt\,.
$$
\end{lemma}

\section{Processes satisfying Assumption L}
\label{sec:special} 

In this section we prove our main theorem, giving a full description
of possible \cadlag\ densities $f_{\BX,T}$ for continuous stationary
processes satisfying Assumption U$_T$ and Assumption L. 

\begin{theorem} \label{t:construction}
Let $\BX = (X(t),\, t\in\bbr)$ be a stationary sample continuous
process, satisfying Assumption U$_T$ and Assumption L. Then the
restriction of the law $F_{\BX,T}$ of the unique location of the
supremum of the process in $[0,T]$ to the interior $(0,T)$ of the interval is absolutely
continuous. The density $f_{\BX,T}$ has a \cadlag\ version with the
following properties: 

(a) \ The density has a bounded variation on $(0,T)$, hence the limits 
$$
f_{\BX,T}(0+)=\lim_{t\to 0} f_{\BX,T}(t) \ \text{and} \ f_{\BX,T}(T-)=\lim_{t\to T}
f_{\BX,T}(t) 
$$
exist and are finite. Furthermore, 
\begin{equation} \label{e:TV}
TV_{(0, T)}(f_{\BX,T}) \leq f_{\BX,T}(0+) + f_{\BX,T}(T-)\,.
\end{equation}

(b) \ The density is bounded away from zero. That is,
\begin{equation} \label{e:density.lower.bound}
\inf_{0<t<T}f_{\BX,T}(t) >0\,.
\end{equation}

(c) \ Either $f_{\BX,T}(t)=1/T$ for all $0<t<T$, or $\int_0^T 
f_{\BX,T}(t)\, dt<1$. 

\bigskip

Moreover, if $f$ is a nonnegative \cadlag\ function satisfying (a)-(c)
above, then there is a stationary sample continuous
process $\BX$, satisfying Assumption U$_T$ and Assumption L, such that
$f$ is the density in the interior $(0,T)$ of the unique location of the
supremum of the process in $[0,T]$. 
\end{theorem}
\begin{proof}
The existence of a \cadlag\ density with properties (a)-(c) in
the statement of the theorem is an immediate consequence of the 
statements of Theorems 3.1 and 3.3 in \cite{samorodnitsky:yi:2011a}. 
We proceed to show the converse part of the theorem. If
$f_{\BX,T}(t)=1/T$ for all $0<t<T$, then a required example is
provided by a single wave periodic stationary Gaussian process with
period $T$, so 
we need only to consider the second possibility in property (c).  We
start with the case where the candidate density $f$ is a piecewise
constant function of a special form. 

We call a finite collection $(u_i,v_i), \, i=1,\ldots, m$ of
nonempty open
subintervals of $(0,T)$ {\it a proper collection of blocks} if for any
$i,j=1,\ldots, m$ there are only 3 possibilities:
either $(u_i,v_i)\subseteq (u_j,v_j)$, or $(u_j,v_j)\subseteq
(u_i,v_i)$, or $[u_i,v_i]\cap [u_j,v_j]=\emptyset$. If $u_i=0, \,
v_i=T$, we call $(u_i,v_i)$ a base block. If $u_i=0, \,
v_i<T$, we call $(u_i,v_i)$ a left block. If $u_i>0, \,
v_i=T$, we call $(u_i,v_i)$ a right block. If $u_i>0, \,
v_i<T$, we call $(u_i,v_i)$ a central block. We start with
constructing a stationary process as required in the theorem when the
candidate density $f$ satisfies requirements (a)-(c) of the theorem
and has the form 
\begin{equation} \label{e:f.proper}
f(t) = \frac{1}{HT}\sum_{i=1}^m \one_{[u_i,v_i)}(t), \ 0<t<T
\end{equation}
for some proper collection of blocks, 
with the obvious convention at the endpoints 0 and $T$, for some 
$H>1$. 
 Observe that for functions of the type \eqref{e:f.proper}, 
requirement (b) of the theorem is equivalent to requiring that there
is at least one base block, and requirement (a) is equivalent to requiring that the
number of the central blocks does not exceed the number of the base
blocks. Finally, (the second case of) property (c) is equivalent to
requiring that
\begin{equation} \label{e:W}
d=\frac{1}{m} \left( HT - \sum_{i=1}^m (v_i-u_i)\right)>0\,.
\end{equation}

We will construct a stationary
process by a uniform shift of a periodic deterministic function over
its period. Now, however, the period will be equal to $HT>T$. We
start, therefore, by defining a 
deterministic continuous function $( x(t),\, 0\leq t\leq HT)$ with
$x(0)=x(HT)$, which we then extend by periodicity to the entire
$\bbr$. Let $B\geq 1$ be the number of the base blocks in the
collection. We partition the entire collection of blocks into $B$
subcollection which we call {\it components} by assigning each base
block to one component, assigning to each component at most one
central block, and assigning the left and right blocks to components
in an arbitrary way. For $j=1,\ldots, B$ we denote by
\begin{equation} \label{e:Lj}
L_j = d\bigl( \text{the number of blocks in the $j$th component} \bigr)
\end{equation}
$$+\, 
\text{the total length of the blocks in the $j$th
  component}\,.
$$
We set $x(0)=2$. Using the blocks of the first component we will
define the function $x$ on the interval $(0,L_1]$ in such a way that
$x(L_1)=2$. Next, using the blocks of the second component we will
define the function $x$ on the interval $(L_1,L_1+L_2]$ in such a way that
$x(L_1+L_2)=2$, etc. Since 
$$
\sum_{j=1}^B L_j = dm + \sum_{i=1}^m (v_i-u_i) = HT\,,
$$
this construction will terminate with a function $x$ constructed on
the entire interval $[0,HT]$ with $x(HT)=2=x(0)$, as desired. 

We proceed, therefore, with defining the function $x$ on an interval
of length $L_j$ using the blocks of the $j$th component. For
notational simplicity we will take $j=1$ and define $x$ on the
interval $[0,L_1]$ using the blocks of the first component. The
construction is slightly different depending on whether or not the
component has a central block, whether or not it has any left blocks,
and whether or not it has any right blocks. If the component has
$l\geq 1$ left blocks, we will denote them by $(0,v_j)$, $j=1,\ldots,
l$. If the component has $r\geq 1$ right blocks, we will denote them
by $(u_j,T)$, $j=1,\ldots, r$. If the component has a central block,
we will denote it by $(u,v)$. We will construct the function $x$ by
defining it first on a finite number of special points and then
filling in the gaps in a piecewise linear manner.

Suppose first that the component has a central block, some left blocks
and some right blocks. In this case we proceed as follows.

{\bf Step 1} \ Recall that $x(0)=2$ and set
$$
x\left( jd + \sum_{i=1}^{j-1}v_i\right) = 
x\left( jd + \sum_{i=1}^{j}v_i\right) = 2-2^{j-l},\ j=1,\ldots, l\,.
$$
Note that the last point obtained in this step is 
$x\bigl( ld + \sum_{i=1}^{l}v_i\bigr) = 1$. 

{\bf Step 2} \ Set 
$$
x\left( (l+1)d + \sum_{i=1}^{l}v_i\right) 
= x\left( (l+1)d + \sum_{i=1}^{l}v_i+v\right)
$$
$$
=
x\left( (l+1)d + \sum_{i=1}^{l}v_i+v+T-u\right)
=\frac12\,.
$$

{\bf Step 3} \ Set 
$$
x\left( (l+j+1)d + \sum_{i=1}^{l}v_i+v+T-u +
  \sum_{i=1}^{j-1}(T-u_j)\right)
$$
$$
= x\left( (l+j+1)d + \sum_{i=1}^{l}v_i+v+T-u +
  \sum_{i=1}^{j}(T-u_j)\right)
$$
$$
= 2-2^{-(j-1)}, \ j=1,\ldots, r\,.
$$
Note that the last point obtained in this step is
$$
x\left( (l+r+1)d + \sum_{i=1}^{l}v_i+v+T-u +
  \sum_{i=1}^{r}(T-u_j)\right)= 2-2^{-(r-1)}\,.
$$

{\bf Step 4} \ We add just one more point at distance $d$ from the
last point of the previous step by setting
$$
x\left( (l+r+2)d + \sum_{i=1}^{l}v_i+v+T-u +
  \sum_{i=1}^{r}(T-u_j)\right)= 2\,.
$$
Note that this point coincides with $L_1$ as defined in \eqref{e:Lj}. 

\bigskip

If the component has no left blocks, then Step 1 above is skipped, and
Step 2 becomes the initial step with 
$$
x(d) = x(d+v) = x(d+v+T-u) =\frac12\,.
$$

If the component has no right blocks, then Step 3 above is skipped,
and at Step 4 we add the distance $d$ to the final point of Step 2,
that is we set 
$$
x\left( (l+2)d + \sum_{i=1}^{l}v_i+v+T-u\right) = 2\,.
$$

If the component has no central block, then Step 2 is skipped, but we
do add the distance $T$ to the last point of Step 1. That is, the
first point obtained at Step 3 becomes 
$$
x\left( (l+1)d + \sum_{i=1}^{l}v_i+T\right)=1\,,
$$
if there are any left blocks, with the obvious change if
$l=0$. Finally, if there is neither central block, nor any right
blocks, then both Step 2 and Step 3 are skipped, and Step 4 just adds
$d+T$ to the last point of Step 1, i.e. it becomes
$$
x\left( (l+1)d + \sum_{i=1}^{l}v_i +T  \right)= 2\,, 
$$
once again with the obvious change  if $l=0$. It is easy to check that
in any case Step 4 sets $x(L_1)=2$, with $L_1$ as defined in
\eqref{e:Lj}. In particular, $L_1> T$. 

Finally, we specify the piecewise linear rule by which we complete the
construction of the function $x$ on the interval $[0,L_1]$. The
function has been defined on a finite set of points and we proceed
from left to right, starting with $x(0)=2$, to fill the gap between
one point in the finite set and the adjacent point from the right,
until we reach $x(L_1)=2$. By the construction, there are pairs of
adjacent points in which the values of $x$ coincide, and  pairs 
of adjacent points in which the values of $x$ are different. In most
cases only adjacent points at the distance $d$ have equal values of
$x$, but if, e.g.  a central block is missing, then at a pair of adjacent points
at a distance $T$, or $d+T$, the values of $x$ coincide as well. 

In any case, if the values of $x$ at two adjacent points are
different, we define the values of $x$ between these two points by
linear interpolation. If the values of $x$ at two adjacent points,
say, $a$ and 
$b$ with $a<b$, are equal to, say, $y$ we define the function $x$
between these two points by
$$
x(t) = \max\bigl( y-(t-a)/d, \, y-(b-t)/d\bigr)
$$
provided the value at the midpoint, $y-(b-a)/2d\geq -1$. If this lower
bounds fails, we define the values of $x$ between the points $a+dy$
and $b-dy$ by
$$
x(t) = \max\bigl( -\tau( t-(a+dy)), \, -\tau(( b-dy)-t)\bigr)\,,
$$
for an arbitrary $\tau>0$ such that both $\tau\leq 1/d$ and the value at
the midpoint, $-\tau\bigl( (b-a)/2-dy\bigr)\geq -1$. The reason for
this slightly cumbersome definition is the need to ensure that $x$ is
nowhere constant, while keeping the lower bound of $x$ and its
Lipschitz constant under control. We note, at this point, that, since
in all cases $b-a\leq T+d$, we can choose, for a fixed $T$, the value
of $\tau$  so that $\tau\geq \tau_d>0$, where the constant $\tau_d$
stays bounded away from zero for $d$ in a compact interval. 

Now that we have defined a periodic function $(x(t),\, t\in\bbr)$ with
period $HT$, we define a stationary process $\BX$ by $X(t)= x(t-U)$,
$t\in\bbr$, where $U$ is uniformly distributed between 0 and $HT$. The
process is, clearly, sample continuous and satisfies Assumption L. We
observe, further, that, if the supremum in the interval $[0,T]$ is
achieved in the interior of the interval, then it is achieved at a
local maximum of the function $x$. If the value at the local maximum
is equal to 2, then 
it is due to an endpoint of a
component, and, since the contribution of any component  has length
exceeding $T$, this supremum is unique. If the value at the
local maximum is smaller than 2, then that local maximum is separated
from the nearest local maximum with
the same value of $x$ by at least the distance induced by Step 2,
which $T$. Consequently, in this case the supremum over $[0,T]$ is
unique as well. Similarly, if the supremum is achieved at one of the
endpoints of the interval, it has to be unique as well, on a set of
probability 1. Therefore, the process $\BX$ satisfies Assumption
U$_T$. 

We now show that for the process $\BX$ constructed above, the density
$f_{\BX,T}$ coincides with the function $f$ given in
\eqref{e:f.proper}, with which the construction was performed. According to
the above analysis, we need to account for the contribution of each
local maximum of the function $x$ over its period to the density
$f_{\BX,T}$. The local maxima may appear in Step 1 of the
construction, and then they are due to left blocks. They may
apear in Step 3 of the construction, and then they are due to 
right blocks. They may appear Step 2 of the
construction, and then they are due to central blocks. Finally, the
points where $x$ has value 2 are always local maxima. We will see
that they are due to base blocks. We start with the latter local
maxima. Clearly, each such local maximum is, by periodicity, equal to
one of the $B$ 
values, $\sum_{j=1}^{i} L_j-HT, \, i=1,\ldots, B$. The $i$th of these
points becomes the global maximum of $\BX$ over $[0,T]$ if and only if 
$$
U\in \left( HT-\sum_{j=1}^{i} L_j, \, (H+1)T-\sum_{j=1}^{i}
  L_j\right)\,,
$$
and the global maximum is then located at the point
$\sum_{j=1}^{i} L_j - HT+U$. Therefore, the contribution of each such
local maximum to the density is $1/HT$ at each $0<t<T$, and overall
the points where $x$ has value 2 contribute to $f_{\BX,T}$
\begin{equation} \label{e:contribute.base}
f_{\rm base}(t) = \frac{B}{HT}, \ 0<t<T\,.
\end{equation}

Next, we consider the contribution to $f_{\BX,T}$ of the local maxima
due to left blocks. For simplicity of notation we consider only the
left blocks in the first component. Then the local maximum due to the
$j$th left block is at the point $jd + \sum_{i=1}^{j}v_i$. 
As before, we need to check over what interval of the values of $U$
this local maximum becomes the global maximum of $\BX$ over
$[0,T]$. The relevant values of $U$ must be such that the time
interval $\bigl(jd + \sum_{i=1}^{j-1}v_i, \, jd +
\sum_{i=1}^{j}v_i\bigr)$ is shifted to cover the origin, and this
corresponds to an interval of length $v_j$ of the values of $U$. The
shifted local maximum itself will then be located within the interval
$(0,v_j)$, which contributes $1/HT$ at each $0<t<v_j$. Overall, the local maxima
due to left blocks contribute to $f_{\BX,T}$
\begin{equation} \label{e:contribute.left}
f_{\rm left}(t) = \frac{1}{HT}\sum_{\rm left\ blocks}
\one_{(0,v_i)}(t), \ 0<t<T\,. 
\end{equation}
Similarly,  the local maxima
due to right blocks contribute to $f_{\BX,T}$
\begin{equation} \label{e:contribute.right}
f_{\rm right}(t) = \frac{1}{HT}\sum_{\rm right\ blocks}
\one_{(u_i,T)}(t), \ 0<t<T\,. 
\end{equation}

Finally, we consider the central blocks. If the first component has a
central block, then the local maximum due to the central block is at
the point $(l+1)d + \sum_{i=1}^{l}v_i+v$. Any value of
$U$ that makes this local maximum the global maximum over $[0,T]$ must
be such that the time interval $\bigl( (l+1)d + \sum_{i=1}^{l}v_i, \,
(l+1)d + \sum_{i=1}^{l}v_i+v\bigr)$ is shifted to cover the
origin. Furthermore, that value of $U$ must also be such that the time
interval $\bigl( (l+1)d + \sum_{i=1}^{l}v_i+v, \,
(l+1)d + \sum_{i=1}^{l}v_i+v+T-u\bigr)$ is shifted to cover the right
endpoint $T$. If we think of shifting the origin instead of shifting
$x$, the origin will have to be located in the interval $\bigl( (l+1)d + \sum_{i=1}^{l}v_i, \,
(l+1)d + \sum_{i=1}^{l}v_i+v-u\bigr)$. This corresponds to a set of
values of $U$ of measure $v-u$, and the shifted local maximum will
then be located within the interval $(u,v)$, which contributes $1/HT$
at each $u<t<v$ to the density. Overall, the local maxima
due to central blocks contribute to $f_{\BX,T}$
\begin{equation} \label{e:contribute.central}
f_{\rm central}(t) = \frac{1}{HT}\sum_{\rm central\ blocks}
\one_{(u,v)}(t), \ 0<t<T\,. 
\end{equation}
Since
$$
f_{\BX,T}(t) = f_{\rm base}(t) + f_{\rm left}(t) + f_{\rm right}(t)+
f_{\rm central}(t), \ 0<t<T\,,
$$
we conclude by \eqref{e:contribute.base} -
\eqref{e:contribute.central} that $f_{\BX,T}$ indeed coincides with
the function $f$ given in \eqref{e:f.proper}. Therefore, we have
proved the converse part of the theorem in the case when the candidate
density $f$ is of the form \eqref{e:f.proper}. 

We now prove the converse part of the theorem for a general $f$
with properties (a)-(c) in the statement of the theorem. Recall
that we need only to treat the second possibility in property (c).  In
order to construct a stationary process $\BX$ for which $f_{\BX,T}=f$,
we will approximate the candidate density $f$ by functions of
the form \eqref{e:f.proper}. Since we will need to deal with 
convergence of a sequence of continuous stationary processes we have
just constructed in the case when the candidate density is of the form
\eqref{e:f.proper}, we record, at this point, several properties of
the stationary periodic process $X(t)= x(t-U)$, $t\in\bbr$ constructed
above. 

{\bf Property 1} \ {\it The process $\BX$ is uniformly bounded: $-1\leq
X(t)\leq 2$ for all $t\in\bbr$.}

\medskip

{\bf Property 2} \ {\it The process $\BX$ is Lipschitz continuous, and
  its Lipschitz constant does not exceed $3/2d$.}

\medskip

{\bf Property 3} \ {\it The process $\BX$ is differentiable except at
  countably many points, at which $\BX$ has left and right
  derivatives. On the set $D_0=\{ t:\, X(t)>0\}$ the derivatives
  satisfy 
$$
\left| X^\prime (t)\right| \geq \frac{1}{2^Nd}
$$
(where the bound applies to both left and right derivatives if $t$ is
not a differentiability point). Here $N$ is the bigger of the largest
number of left blocks any component
has, and the largest number of the right blocks any component
has. Similarly, on the set $D_1=\{ t:\, X(t)\leq 0\}$ the derivatives
  satisfy 
$$
\left| X^\prime (t)\right| \geq \tau_d\,,
$$
where $\tau_d>0$ stays bounded away from zero for $d$ in a compact interval.
}

\medskip

{\bf Property 4} \ {\it The distance between any  two local maxima of
  $\BX$ cannot be smaller than $d$. At its local maxima, $\BX$ takes
  values in a finite set of at most $N+3$ elements. Moreover, the
  absolute difference in the 
  values of the process $\BX$ in two local maxima in the interval
  $(0,T)$ is at least $2^{-N}$, where $N$ is as above.} 

All these properties follow from the corresponding properties of the function
$x$ by considering the possible configuration of the blocks in a
component.

We will now construct a sequence of approximations to a candidate
density $f$ as above. Let $n=1,2,\ldots$. It follows from the general
properties of \cadlag\ functions (see e.g. \cite{billingsley:1999})
that there is a finite partition $0=t_0<t_1<\ldots<t_k=T$ of the
interval $[0,T]$ such that
\begin{equation}\label{e:unif.cty.int}
|f(s)-f(t)|\leq \frac{1}{nT} \ \text{for all $t_i\leq s,t<t_{i+1}$,
  $i=0,\ldots, k-1$.}
\end{equation}
We define a piecewise constant function $\tilde f_n$ on $(0,T)$ by
setting, for each $i=1,\ldots, k$, the value of $\tilde f_n$ for
$t_{i-1}\leq t<t_i$ to be 
$$
\tilde f_n(t) = \frac{1}{knT}\max\left\{ j=0,1,\ldots:\, f(s)\geq
  \frac{j}{knT}\ \text{for all $t_{i-1}\leq s<t_i$}\right\}\,.
$$
By definition and \eqref{e:unif.cty.int} we see that
\begin{equation}\label{e:diff.approx}
f(t)-\frac{2}{nT}\leq \tilde f_n(t)\leq f(t), \ 0<t<T\,.
\end{equation}
Next, we notice that for every
$i=1,\ldots,k-1$ there are
points $s_i\in (t_{i-1},t_i)$ and $s_{i+1}\in (t_i,t_{i+1})$ such that
$$
\bigl| f(s_i)-f(s_{i+1})\bigr| \geq \bigl| \tilde f_n(t_i-)-\tilde
f_n(t_i)\bigr| - \frac{1}{knT}\,.
$$
Therefore,
\begin{equation}\label{e:TV.approx}
TV_{(0, T)}(\tilde f_n)\leq TV_{(0, T)}(f) + \frac{1}{nT}\,.
\end{equation}
We now define
$$
f_n(t) = \tilde f_n(t) + \frac{1}{nT}\, \ 0<t<T\,.
$$
Clearly, the function $ f_n$ is \cadlag, has bounded variation
on $(0,T)$  and is bounded away from
zero. By \eqref{e:TV.approx},  $ f_n$ also satisfies
\eqref{e:TV} since $f$ does. Finally, since $\int_0^T 
f_{\BX,T}(t)\, dt<1$, we see by \eqref{e:diff.approx} that, for all
$n$ large enough, $\int_0^T 
 f_{\BX,T}(t)\, dt<1$ as well. Therefore, for such $n$ the
function $ f_n$ has properties (a)-(c) in the statement of the
theorem, and in the sequel we will only consider $n$ large as
above. We finally notice that $ f_n$ takes finitely many
different values, all of which are in the set $\{ j/knT, \,
j=1,2,\ldots\}$. Therefore, $ f_n$ can be written in the form
\eqref{e:f.proper}, with $H=kn$. Indeed, the blocks can be built by
combining into a block all neighboring intervals where
the value of $ f_n$ is the smallest, subtracting $1/knT$ from
the value of $ f_n$ in the constructed block and iterating the
procedure.  

We have already proved that for any function of the type
\eqref{e:f.proper} there is a stationary process required in the
statement of the theorem. Recall that a construction of this
stationary process depends on assignment of blocks in a proper
collection to components, and we would like to make sure that no
component has ``too many'' left or right blocks. To achieve this, we
need to distribute the left and right blocks as evenly as possible
between the components. Two observations are useful here. First of
all, it follows from the definition of $f_n$ and \eqref{e:f.proper}
that 
$$
\frac{1}{k_nnT} (L_n+B_n) = 
f_n(0+)\leq f(0+)+ \frac{1}{k_nnT} \leq f(0+)+ 1
$$
for $n$ large enough (we are writing $k_n$ instead of $k$ to 
emphasize the dependence of $k$ on $n$), where $L_n$ and $B_n$ are the
numbers of the the left and base blocks in the $n$th collection. On
the other hand, similar considerations tell us that
$$
\frac{1}{k_nnT} B_n = \inf_{0<t<T} f_n(t) \geq 
\inf_{0<t<T} f(t) - \frac{2}{nT} 
\geq \frac12 \inf_{0<t<T} f(t)\,,
$$
once again for $n$ large enough, where we have used property (b) of
$f$. Therefore, for such $n$,
\begin{equation} \label{e:few.left}
\frac{L_n}{B_n} \leq 2\frac{f(0+)+ 1}{\inf_{0<t<T} f(t)}\,,
\end{equation}
and the right hand side  is a finite quantity depending on $f$, but
not on $n$. Performing a similar analysis for the right blocks, and
recalling that we are distributing the left and right blocks as evenly as possible
between the components, we see that there is a number $\Delta_f\in
(0,\infty)$ such that for all $n$ large enough, no component in the
$n$th collection has more than $\Delta_f$ left blocks or $\Delta_f$
right blocks. 

We will also need bounds on the important parameter $d=d_n$ appearing
in the construction of a stationary process corresponding to functions
of the type \eqref{e:f.proper}; these bounds do not depend on a
particular way we assigns blocks to different components. Recall that 
\begin{equation} \label{e:dn}
d_n = \frac{k_nnT}{m_n}\left( 1- \int_0^T f_n(t)\, dt\right)\,,
\end{equation}
where $m_n = B_n+L_n+R_n+C_n$ (in the obvious notation) is the total
number of blocks in the $n$th collection. Since
$$
\frac{1}{k_nnT}\bigl( B_n + \max(L_n,R_n,C_n)\bigr) =
\sup_{0<t<T}f_n(t), \ \ \frac{1}{k_nnT}B_n = \inf_{0<t<T}f_n(t)\,,
$$
we see that 
\begin{equation} \label{e:dn1}
\inf_{0<t<T}f_n(t)\leq \frac{1}{k_nnT}   m_n \leq 3
\sup_{0<t<T}f_n(t)\,.
\end{equation}
We also know by the uniform convergence that $\int_0^T f_n\to \int_0^T
f$. Therefore, by \eqref{e:dn} and \eqref{e:dn1} we obtain that, for
all $n$ large enough, 
\begin{equation} \label{e:dn.bounds}
\frac{1- \int_0^T f(t)\, dt}{4 \sup_{0<t<T}f(t)}\leq
d_n\leq \frac{2}{\inf_{0<t<T}f(t)}\,. 
\end{equation}

An immediate conclusion is the following fact. By
construction, the distribution of $X_n(0)$ is absolutely
continuous; let $g_n$ denote the right continuous version of its
density.  Since $\BX_n$ is obtained by uniform shifting of a
piecewise linear periodic function with period $H_nT$, the value of the
density $g_n(v)$ at each point $v$ times the length of the period   does not exceed
the total number of the linear 
pieces in a period divided by the smallest absolute slope of any
linear piece. The former does not exceed $2m_n$, and by {\bf Property
  3} and the above, the latter cannot be smaller than
$$
\min\left(\frac{1}{2^{\Delta_f}d_n}, \, \tau_{d_n}\right)\,.
$$
Since, by \eqref{e:dn.bounds}, $d_n$ is uniformly bounded from above, 
we 
conclude, for some finite positive constant $c=c(f)$, 
$g_n(v) \leq c(f) m_n/H_n$. Further, by the definition of $d_n$, 
$$
m_nd_n = H_nTP\bigl( \tau_{\BX_n,T}\in \{ 0,T\}\bigr)
\leq H_nT\,.
$$
Once again, since by \eqref{e:dn.bounds}, $d_n$ is uniformly bounded
from below, we conclude that 
\begin{equation} \label{e:density,zero}
g_n(v) \ \text{is uniformly bounded in $v$ and $n$.}
\end{equation}

Let $\BX_n$ be the stationary process corresponding to $f_n$
constructed above. We view $\BX_n$ as a random element of the space
$\bbc$ of continuous functions on $\bbr$ which we endow with the
metric
$$
\rho(\bx, \by)= \sum_{m=1}^\infty 2^{-m} \bigl( \sup_{|t|\leq
  m}|x(t)-y(t)|\bigr)\,.
$$
Let $\mu_n$ be the law of $\BX_n$ on $\bbc$, $n=1,2,\ldots$ (but large
enough, as needed). By {\bf Property 1} and {\bf Property 2} of the processes $\BX_n$ and
the lower bound in \eqref{e:dn.bounds}, these processes are uniformly bounded
and equicontinuous. Therefore, by Theorem 7.3 in
\cite{billingsley:1999}, for every fixed $m=1,2,\ldots$ the
restrictions of the measures $\mu_n$ to the interval $[-m,m]$ form a
tight family of probability measures. Let $n_{1j}\to\infty$ be a
sequence positive integers such that the restrictions of
$\mu_{n_{1j}}$ to $[-1,1]$ converge weakly to a probability measure
$\nu_1$ on $C([-1,1])$.  Inductively define for
$m=2,3,\ldots$ $n_{mj}\to\infty$ to be a subsequence of the sequence 
$n_{m-1,j}\to\infty$ such that the restrictions of
$\mu_{n_{mj}}$ to $[-m,m]$ converge weakly to a probability measure
$\nu_m$ on $C([-m,m])$. Then the ``diagonal'' sequence of measures
$\bigl(\mu_{n_{jj}},\, j=1,2,\ldots \bigr)$
is such that the restrictions of these measures to each interval
$[-m,m]$ converge weakly to $\nu_m$ on $C([-m,m])$. By the Kolmogorov
existence theorem, there is a (cylindrical) probability measure $\nu$ on
functions on $\bbr$ whose restrictions to each interval $[-m,m]$
coincide with $\nu_m$ (considered now as a cylindrical measure). Since
each probability measure $\nu_m$ is supported by $C([-m,m])$, the
measure $\nu$ itself is supported by functions in $\bbc$. By
construction, the measure $\nu$ is shift invariant. If $\BX$ is the
canonical stochastic process defined on $\bigl(\bbc, \nu\bigr)$, then $\BX$ is a
sample continuous stationary process. In the remainder of the proof we
will show that $\BX$ satisfies Assumption L and Assumption U$_T$, and
that $f_{\BX,T}=f$.

We start with proving that Assumption L holds for $\BX$. It is,
clearly, enough to prove that, on a set of probability 1,
\begin{equation} \label{e:no.accum}
\text{any two local maxima of $\BX$ are at least}
\ \theta:=\frac{1-\int_0^Tf(t)\, dt}{5 \sup_{0<t<T}f(t)}\ \text{apart.}
\end{equation}
Suppose that \eqref{e:no.accum} fails. Then there is $m$ sucht that,
on an event of positive probability, two local maxima of $\BX$ closer
than $\theta$ exist in the time interval $[-m,m]$. Recall that a
subsequence of the sequence of the (laws of) $\BX_n$ converges weakly
in the uniform topology on $C([-m,m])$ to the (law of) $\BX$. For
notational simplicity we will identify that subsequence with the
entire sequence $(\BX_n)$. By the
Skorohod representation theorem (Theorem 6.7 in
\cite{billingsley:1999}), we may define the processes  $(\BX_n)$ on
some probability space so that $\BX_n\to\BX$ a.s. in $C([-m,m])$. Fix
$\omega$ for which this convergence holds, and for which $\BX$ has two
local maxima closer
than $\theta$ exist in the time interval $[-m,m]$.  It straightforward
to check that the uniform convergence and {\bf Property 3} above imply 
that for all $n$
large enough, the processes $\BX_n$ will have two local maxima closer
than $5\theta/4$. This is, of 
course, impossible, due to {\bf Property 4} and
\eqref{e:dn.bounds}. The resulting contradiction proves that $\BX$
satisfies Assumption L. 

Next, we prove that Assumption U$_T$ holds for $\BX$.  Since the
process $\BX$ satisfies Assumption L, by Lemma 2.2 in
\cite{samorodnitsky:yi:2011a},  it has
finitely many local maxima in the interval $(0,T)$ (in fact, by
\eqref{e:no.accum}, it cannot have more than $\lceil T/\theta\rceil$
local maxima). Clearly, the values of $\BX$ at the largest local
maximum and the second largest local maximum  (if any) are well
defined random variables. We denote by 
$(M_1,M_2)$ the largest and the second largest among $X(0)$, $X(T)$
and the values of $\BX$ at the largest local
maximum and the second largest local maximum  (if any). The fact that
Assumption U$_T$ holds for $\BX$ will follow once we prove that
\begin{equation} \label{e:M1M2}
P\bigl( M_1=M_2\bigr)=0\,.
\end{equation}
We proceed similarly to the argument in the proof of Assumption L. We
may assume that $\BX_n\to\BX$ a.s. in $C[0,T]$. Fix $\omega$ for which
this convergence holds. The uniform convergence and {\bf Property 3}
of the processes $(\BX_n)$, together with the uniform upper bound on
$d_n$ in \eqref{e:dn.bounds}, show that, for every local maximum
$t_\omega$ of $\BX$ in the interval $(0,T)$ and any $\delta>0$, there
is $n(\omega,\delta)$ such that for all $n> n(\omega,\delta)$, the
process $\BX_n$ has a local maximum in the interval $(t_\omega-\delta,
t_\omega+\delta)$. This immediately implies that
$$
M_1-M_2\geq \limsup_{n\to\infty} \bigl( M_1^{(n)} - M_2^{(n)}\bigr)
$$
a.s., where the random vector $(M_1^{(n)} ,M_2^{(n)} )$ is defined for
the 
process $\BX_n$ in the same way as the random vector $(M_1,M_2)$ is defined for the
process $\BX$, $n=1,2,\ldots$. In particular, for any $\vep>0$,
\begin{equation} \label{e:far.apart}
P\bigl( M_1-M_2<\vep\bigr) \leq \limsup_{n\to\infty} P\bigl( M_1^{(n)}
- M_2^{(n)}<\vep\bigr) \,.
\end{equation}

As a first step, notice that, by {\bf Property 4} of the processes
$(\BX_n)$, for any $\vep<\Delta_f$, 
\begin{equation} \label{e:apart.1}
P\Bigl( M_1^{(n)} - M_2^{(n)}<\vep, 
\end{equation}
$$
\text{both $M_1^{(n)}$ and
  $M_2^{(n)}$ achieved at local maxima}\Bigr)=0
$$
for each $n$. Next, since by {\bf Property 4}, at its local maxima the
process $\BX_n$ can take at most $\Delta_f+3$ possible values, we
conclude by \eqref{e:density,zero} that for all $\vep>0$, 
\begin{equation} \label{e:apart.2}
P\Bigl( M_1^{(n)} - M_2^{(n)}<\vep, \ \text{one of $M_1^{(n)}$} 
\end{equation}
$$
\text{ and
  $M_2^{(n)}$ is achieved at a local maximum, and one at an
  endpoint}\Bigr) \leq c_f\vep\,,
$$
for some $c_f\in (0,\infty)$. Finally, we consider the case when both 
$M_1^{(n)}$ and $M_2^{(n)}$ are achieved at the endpoints of the
interval $[0,T]$. In that case, it is impossible that $\BX_n$ has a
local maximum in $(0,T)$, since that would force time 0 to belong
to one of the decreasing linear pieces of the process due to left
blocks, and time $T$ to belong one of the increasing linear pieces of
the process due to right blocks. By construction, the distance between
any two points belonging to such intervals is larger than $T$. That
forces $X_n(t), 0\leq t\leq T$ to consist of at most two linear
pieces. By {\bf Property 3} of the process $\BX_n$, in order to
achieve $|X_n(0)-X_n(T)|\leq \vep$, each block of the proper
collection generating $\BX_n$ contributes at most an interval of
length $\vep /\min(1/(2^\Delta d_n),\tau_{d_n})$ to the set of
possible shifts $U$. Recall 
that there are $m_n$ blocks in the collection. By the uniform bounds
\eqref{e:dn.bounds} we conclude that for all $\vep>0$,
\begin{equation} \label{e:apart.3}
P\Bigl( M_1^{(n)} - M_2^{(n)}<\vep, 
\end{equation}
$$
\text{ $M_1^{(n)}$  and
  $M_2^{(n)}$ achieved at the  endpoints}\Bigr) 
$$
$$
\leq \vep \frac{m_n}{H_nT} \frac{1}{\min(1/(2^\Delta d_n),\tau_{d_n})} 
\leq \vep \frac{1}{d_n\min(1/(2^\Delta d_n),\tau_{d_n}) }
\leq  \tilde c_f\vep\,,
$$
for some $\tilde c_f\in (0,\infty)$. 

Combining \eqref{e:far.apart}, \eqref{e:apart.1}, \eqref{e:apart.2}
and \eqref{e:apart.3} we see that for all $\vep>0$ small enough, 
$$
P\bigl( M_1-M_2<\vep\bigr) \leq (c_f+\tilde c_f)\vep\,.
$$
Letting $\vep\downarrow 0$ we obtain \eqref{e:M1M2}, so that the
process $\BX$ satisfies Assumption U$_T$. 

It is now a simple manner to finish the proof of the theorem. Assume,
once again, that $\BX_n\to\BX$ a.s. in $C[0,T]$. Fix $\omega$ for which
this convergence holds, and both $\BX$ and each $\BX_n$ have a unique
supremum in the interval $[0,T]$. It follows from the uniform
convergence that $\tau_{\BX_n,T}\to \tau_{\BX,T}$ as
$n\to\infty$. Therefore, we also have that $\tau_{\BX_n,T}\Rightarrow
\tau_{\BX,T}$ (weakly). However, by construction, $f_n(t)\to f(t)$ for
every $0<t<T$. This implies that $f$ is the density of $\tau_{\BX,T}$,
and the proof of the theorem is complete. 
\end{proof}

\section{Long intervals} \label{sec:long.intervals} 

In spite of the broad range of
possibilities for the distribution of the supremum location shown in
the previous section, it turns out that, when
the length of an interval becomes large, and the process satisfies a
certain strong mixing assumption, uniformity of the distribution of
the supremum location becomes visible at certain scales. We make this
statement precise in this section. 

In this section we allow a stationary process $\BX$ to have upper
semi-continuous, not necessarily continuous, sample
paths. Moreover, we will not generally impose either Assumption U$_T$,
or Assumption L. Without Assumption U$_T$, the supremum may not be
reached at a unique point, so we will work with the leftmost supremum
location defined in Section \ref{sec:assumptions}. 

Recall that a stationary stochastic process $\BX = (X(t),\, t\in\bbr)$
is called strongly mixing (or $\alpha$-mixing, or uniformly mixing) if 
$$
\sup\Bigl\{ \bigl| P(A\cap B)-P(A)P(B)\bigr|:\, A\in
\sigma\bigl( X(s), \, s\leq 0\bigr), \, B\in \sigma\bigl( X(s), \,
s\geq t\bigr)\Bigr\} 
$$
$$
\to 0 \ \text{as $t\to\infty$;}
$$
see e.g. \cite{rosenblatt:1962}, p. 195. Sufficient conditions on the
spectral density of a stationary Gaussian process that guarantee
strong mixing were established in \cite{kolmogorov:rozanov:1960}. 

Let $\BX$ be an upper semi-continuous  stationary process. 
We introduce  a ``tail version'' of the strong mixing assumption, defined as 
follows. 

{\bf Assumption} TailSM: \ there is a function $\varphi:\,
(0,\infty)\to \bbr$ such that
$$
\lim_{t\to\infty}P\bigl( \sup_{0\leq s\leq t}X(s)\geq
\varphi(t)\bigr)= 1
$$
and 
$$
\sup\Bigl\{ \bigl| P(A\cap B)-P(A)P(B)\bigr|:\, A\in
\sigma\bigl( X(s)\one(X(s)\geq \varphi(t)), 
$$
$$
 s\leq 0\bigr), \, B\in
\sigma\bigl( X(s) \one(X(s)\geq \varphi(t)), \, s\geq t\bigr)\Bigr\} 
\to 0 \ \text{as $t\to\infty$.}
$$

\bigskip

It is clear that if a process is strongly mixing, then it also
satisfies Assumption TailSM. The point of the latter assumption
is that we are only interested in mixing properties of the part of the
process ``responsible'' for its large values.  For example, the process
$$
X(t)= \left\{ \begin{array}{ll}
Y(t) & \text{if $Y(t)>1$} \\
Z(t) & \text{if $Y(t)\leq 1$}
\end{array}
\right., \, t\in\bbr\,,
$$
where $\BY$ is a strongly mixing process  such that $P(Y(0)>1)>0$,
and $\BZ$ an arbitrary stationary process such that $P(Z(0)<1 )=1$,
does not have to be strongly mixing, but it clearly satisfies 
  Assumption TailSM with $\varphi\equiv 1$.  

We will impose one more assumption on the stationary processes we
consider in this section. It deals with the size of the largest atom
the distribution of the supremum of the process may have.

{\bf Assumption} A: 
$$
\lim_{T\to\infty} \sup_{x\in\bbr}P\bigl( \sup_{t\in [0,T]}
X(t)=x\bigr) = 0\,.
$$
In Theorem \ref{t:asymp.unif} below Assumption A could be replaced by
requiring Assumption 
U$_T$ for all $T$ large enough. We have chosen Assumption A instead
since for many important stationary stochastic processes the supremum
distribution is known to be atomless anyway; see e.g. \cite{ylvisaker:1965}
for continuous Gaussian processes and \cite{byczkowski:samotij:1986}
for certain stable processes. The following sufficient condition for
Assumption A is also elementary: suppose that the process $\BX$ is
ergodic. If for some $a\in\bbr$, $P \bigl( \sup_{t\in [0,1]}
X(t)=x\bigr)=0$ for all $x>a$ and $P(X(0)>a)>0$, then Assumption A is
satisfied.

\begin{theorem} \label{t:asymp.unif}
Let $\BX = (X(t),\, t\in\bbr)$ be a stationary sample upper semi-continous
process, satisfying Assumption TailSM and Assumption A. The density
$f_{\BX,T}$ of the supremum location satisfies
\begin{equation} \label{e:dens.unif}
\lim_{T\to\infty}\sup_{\vep\leq t\leq 1-\vep} \Bigl| T
f_{\BX,T}(tT)-1\Bigr|=0
\end{equation}
for every $0<\vep<1/2$. In particular, the law 
of $\tau_{\BX,T}/T$ converges weakly to the uniform distribution on
$(0,1)$. 
\end{theorem}
\begin{proof}
It is obvious that \eqref{e:dens.unif} implies weak convergence of the law 
of $\tau_{\BX,T}/T$ to the uniform distribution. We will, however,
prove the weak convergence first, and then use it to derive
\eqref{e:dens.unif}. 

We start with a useful claim that, while having nothing to do with any
mixing by itself, will be useful for us in a subsequent application of
Assumption TailSM. Let $T_n, \, d_n\uparrow\infty$, $d_n/T_n\to 0$ as
$n\to\infty$. We claim that for any $\delta\in (0,1)$, 
\begin{equation} \label{e:d.n}
P\Bigl( \delta T_n-d_n\leq \tau_{\BX,T_n} \leq \delta
T_n+d_n\Bigr)=0\,.
\end{equation}
To see this, simply note that by \eqref{e:upper.bound}, 
the probability in \eqref{e:d.n} is bounded from
above by 
$$
2d_n \sup_{\delta T_n-d_n\leq t\leq \delta T_n+d_n} f_{\BX,T_n} (t)
\leq 2d_n \max\left( \frac{1}{\delta T_n-d_n}, \, \frac{1}{(1-\delta)
    T_n-d_n}\right) 
\to 0
$$
as $n\to\infty$. 

The weak convergence stated in the theorem will follow once we prove
that for any rational number $r\in (0,1)$, we have 
$P\bigl( \tau_{\BX,T}\leq rT\bigr)\to r$ as $T\to\infty$. Let $r=m/k$,
$m,k\in \bbn, \, m<k$ be such a rational number. Consider $T$ large
enough so that $T>k^2$, and partition the interval $[0,T]$ into
subintervals
$$
C_i = \left[ (T+\sqrt{T})\frac{i}{k}, \, (T+\sqrt{T})\frac{i+1}{k} -
  \sqrt{T}\right], \ i=0,1,\ldots, k-1\,,
$$
$$
D_i = \left[ (T+\sqrt{T})\frac{i}{k}-
  \sqrt{T}, \, (T+\sqrt{T})\frac{i}{k} \right], \ i=1,\ldots, k-1\,,
$$
and observe that by \eqref{e:d.n},
$$
P\Bigl( \tau_{\BX,T} \in \bigcup_{i=1}^{k-1}D_i\Bigr) \to 0 \ \text{as
  $T\to\infty$.}
$$
Therefore, 
\begin{equation} \label{e:split.C}
P\bigl( \tau_{\BX,T}\leq rT\bigr) = P\Bigl( \max_{0\leq i\leq m-1}
M_{i,T}\geq \max_{m\leq i\leq k-1}
M_{i,T}\Bigr) + o(1)
\end{equation}
as $T\to\infty$, where $M_{i,T} = \sup_{t\in C_i}X(t)$,
$i=0,1,\ldots. k-1$. 

Let $\varphi$ be the function given in Assumption TailSM. Then
\begin{equation} \label{e:M.V}
P\Bigl( \max_{0\leq i\leq m-1}
M_{i,T}\geq \max_{m\leq i\leq k-1}
M_{i,T}\Bigr) 
\end{equation}
$$
= P\Bigl( \max_{0\leq i\leq m-1}
V_{i,T}\geq \max_{m\leq i\leq k-1}
V_{i,T}\Bigr) + o(1)\,,
$$
where $V_{i,T} = \sup_{t\in C_i}X(t)\one\bigl( X(t)>\varphi(\sqrt{T})\bigr)$,
$i=0,1,\ldots. k-1$. 

Denote by $G_T$ the distribution function of each one of the random
variables $V_{i,T}$, and let $W_{i,T} = G_T(V_{i,T})$, $i=0,1,\ldots,
k-1$. It is clear that 
\begin{equation} \label{e:V.W}
P\Bigl( \max_{0\leq i\leq m-1}
V_{i,T}\geq \max_{m\leq i\leq k-1}
V_{i,T}\Bigr) 
\end{equation}
$$
=P\Bigl( \max_{0\leq i\leq m-1}
W_{i,T}\geq \max_{m\leq i\leq k-1}
W_{i,T}\Bigr) \,.
$$
Notice, further, that by Assumption TailSM, for every $0<w_i<1, \,
i=0,1,\ldots, k-1$, 
\begin{equation} \label{e:SM.cons}
\lim_{T\to\infty} \Bigl| P\Bigl( W_{i,T}\leq w_i, \, i=0,1,\ldots,
k-1\Bigr) - \prod_{i=0}^{k-1} P\Bigl( W_{i,T}\leq w_i\Bigr)\Bigr| =
0\,.
\end{equation}
Let 
$$
D(T) = \sup _{x\in\bbr}P\bigl( \sup_{t\in C_0}
X(t)=x\bigr) + P \bigl(\sup_{t\in C_0} X(t)\leq
\varphi(\sqrt{T})\bigr)\,.
$$
By Assumption A, $D(T)\to 0$ as $T\to\infty$. Since for every $0<w<1$, 
$$
w-D(T)\leq 
P\Bigl( W_{0,T}\leq w\Bigr)\leq w\,,
$$
we conclude by \eqref{e:SM.cons} that the law of the random vector
$\bigl(  W_{0,T},\ldots,  W_{k-1,T}\bigr)$ converges weakly, as
$T\to\infty$, to the law of a random vector $(U_0,\ldots, U_{k-1})$
with independent standard uniform components. Since this limiting law
does not charge the boundary of the set $\{
(w_0,w_1,\ldots, w_{k-1}):$ $ \max_{0\leq i\leq m-1}
w_i\leq \max_{m\leq i\leq k-1}w_i\}$, we conclude by
\eqref{e:split.C}, \eqref{e:M.V} and \eqref{e:V.W} that 
$$
P\bigl( \tau_{\BX,T}\leq rT\bigr) \to P\bigl(\max_{0\leq i\leq m-1}
U_i\geq \max_{m\leq i\leq k-1}U_i\bigr) = m/k=r\,,
$$
and so we have established the weak convergence claim of the
theorem. 

We now prove the uniform convergence of the densities in
\eqref{e:dens.unif}. Suppose that the latter fails for some
$0<\vep<1/2$. There are two possibilities. Suppose first that 
there is $\theta>0$, a sequence $T_n\to\infty$ and
a sequence $t_n\in [\vep,1-\vep]$ such that for every $n$, $ T_n
f_{\BX,T_n}(t_nT_n)\geq 1+\theta$. By compactness we may assume
that $t_n\to t_*\in [\vep,1-\vep]$ as $n\to\infty$. By Lemma
\ref{l:key} and the regularity properties of the density, for every
$n$ and every  $0<\tau,\delta<1$ such that
\begin{equation} \label{e:range}
\bigl( 1-(1-\tau)/t_n\bigr)_+<\delta<\min\bigl( \tau/t_n,1\bigr)
\end{equation}
we have
$$
T_nf_{\BX,(1-\tau)T_n}(t_n(1-\delta)T_n)\geq T_nf_{\BX,T_n}(t_nT_n)\geq
1+\theta\,.
$$
Since $t_n\to t_*$, there is a choice of $0<\tau<1$ such that
\begin{equation} \label{e:tau.theta}
1+\theta>\frac{1}{1-\tau}
\end{equation} 
and, moreover, the range in 
\eqref{e:range} is nonempty for all $n$ large enough. Furthermore, 
we can find $0<a<b<1$ such that
$$
\bigl( 1-(1-\tau)/t_n\bigr)_+<a<b<\min\bigl( \tau/t_n,1\bigr)
$$
for all $n$ large enough. Therefore, for such $n$
$$
(1+\theta)(b-a)\leq \int_a^b
T_nf_{\BX,(1-\tau)T_n}(t_n(1-\delta)T_n)\, d\delta
$$
$$
= \frac{1}{t_n} P \Bigl( \tau_{\BX,(1-\tau)T_n}\in \bigl(
(1-b)t_nT_n,\, (1-a)t_nT_n\bigr)\Bigr) \to \frac{1}{1-\tau}(b-a)
$$
as $n\to\infty$ by the already established weak convergence. This
contradicts the choice \eqref{e:tau.theta} of $\tau$. 

The second way \eqref{e:dens.unif} can fail is that there is $0<\theta<1$, a
sequence $T_n\to\infty$ and a sequence $t_n\in [\vep,1-\vep]$ such
that for every $n$, $ T_n f_{\BX,T_n}(t_nT_n)\leq 1-\theta$. We can
show that this option is impossible as well by appealing, once again,
to Lemma \ref{l:key} and using an argument nearly identical to the 
one described above. Therefore, \eqref{e:dens.unif} holds, and 
 the proof of the theorem is complete. 
\end{proof}

The following corollary is an immediate conclusion of Theorem
\ref{t:asymp.unif}. It shows the uniformity of the limiting
conditional distribution of the location of the supremum given that it
belongs to a suitable subinterval of $[0,T]$.
\begin{corollary} \label{c:cond.unif}
Let $\BX = (X(t),\, t\in\bbr)$ be a stationary sample upper semi-continous
process, satisfying Assumption TailSM and Assumption A. Let $0<a_T\leq
a^\prime_T<b^\prime_T\leq b_T<T$ be such that 
$$
\liminf_{T\to\infty}\frac{a_T}{T}>0, \ \ 
\limsup_{T\to\infty}\frac{b_T}{T}<1, \ \
\lim_{T\to\infty}\frac{b^\prime_T-a^\prime_T}{b_T-a_T}=\theta.
$$
Then 
$$
\lim_{T\to\infty} P\Bigl( \tau_{\BX,T}\in \bigl(
a^\prime_T,b^\prime_T\bigr)\Big| \tau_{\BX,T}\in
\bigl(a_T,b_T\bigr)\Bigr)=\theta\,.
$$
\end{corollary}


\begin{thebibliography}{7}
\expandafter\ifx\csname natexlab\endcsname\relax\def\natexlab#1{#1}\fi

\bibitem[Billingsley(1999)]{billingsley:1999}
{\sc P.~Billingsley} (1999): {\em Convergence of Probability Measures\/}.
\newblock Wiley, New York, 2nd edition.

\bibitem[Byczkowski and Samotij(1986)]{byczkowski:samotij:1986}
{\sc T.~Byczkowski {\rm and} K.~Samotij} (1986): Absolute continuity of stable
  seminorms.
\newblock {\em The Annals of Probability\/} 14:289--312.

\bibitem[Kolmogorov and Rozanov(1960)]{kolmogorov:rozanov:1960}
{\sc A.~Kolmogorov {\rm and} Y.~Rozanov} (1960): On a strong mixing condition
  for a stationary Gaussian process.
\newblock {\em Teoria Veroyatnostei i Primeneniya\/} 5:222--227.

\bibitem[Leadbetter et~al.(1983)Leadbetter, Lindgren and
  Rootz{\'e}n]{leadbetter:lindgren:rootzen:1983}
{\sc M.~Leadbetter, G.~Lindgren {\rm and} H.~Rootz{\'e}n} (1983): {\em Extremes
  and Related Properties of Random Sequences and Processes\/}.
\newblock Springer Verlag, New York.

\bibitem[Rosenblatt(1962)]{rosenblatt:1962}
{\sc M.~Rosenblatt} (1962): {\em Random Processes\/}.
\newblock Oxford University Press, New York.

\bibitem[Samorodnitsky and Shen(2011)]{samorodnitsky:yi:2011a}
{\sc G.~Samorodnitsky {\rm and} Y.~Shen} (2011): Is the location of the
  supremum of a stationary process nearly uniformly distributed?
\newblock Technical report.

\bibitem[Ylvisaker(1965)]{ylvisaker:1965}
{\sc N.~Ylvisaker} (1965): The expected number of zeros of a stationary
  Gaussian process.
\newblock {\em Annals of Mathematical Statistics\/} 36:1043--1046.

\end{thebibliography}
\end{document}